\documentclass[11pt,twoside, reqno]{amsart}
\usepackage{xcolor}
\usepackage{tikz-cd}

\numberwithin{equation}{section}
\usepackage{latexsym}
\usepackage{epsfig}
\usepackage{amsmath}
\usepackage{amsthm}
\usepackage{amsfonts}
\usepackage{exscale}
\usepackage{graphicx}

\newtheorem{cor}{Corollary}
\newtheorem{prop}{Proposition}

\newtheorem{remark}{Remark}
\newtheorem{lemma}{Lemma}
\newtheorem{theorem}{Theorem}

\newtheorem{defn}{Definition}

\newtheorem{assumption}{Assumption}

\newcommand{\sm}{\left(\begin{smallmatrix}}
\newcommand{\esm}{\end{smallmatrix}\right)}
\newcommand{\eps}{\varepsilon}

\newcommand{\supp}{\operatorname{supp}}
\newcommand{\Ran}{\operatorname{ran}}
\newcommand{\codim}{\operatorname{codim}}

\newcommand{\ck}{\mathcal{K}}

\newcommand{\ch}{\mathcal{H}}

\newcommand{\cL}{\mathcal{L}}

\renewcommand{\eps}{\varepsilon}

\newcommand{\C}{\ensuremath{\mathbb{C}}}
\newcommand{\R}{\ensuremath{\mathbb{R}}}
\newcommand{\N}{\ensuremath{\mathbb{N}}}
\newcommand{\Z}{\ensuremath{\mathbb{Z}}}

\newcommand{\lo}{\mathfrak{L_0}}
\newcommand{\go}{\mathfrak{G_0}}
\newcommand{\lp}{\mathfrak{L_1}}
\newcommand{\gp}{\mathfrak{G_1}}

\newcommand{\Span}{\operatorname{span}}
\newcommand{\norm}[1]{\left\Vert#1\right\Vert}

\newcommand{\twovec}[2]{\left(\begin{array}{c} #1 \\  #2 \end{array}\right)}

\newcommand{\la}{\left\langle}
\newcommand{\ra}{\right\rangle}
\newcommand{\llangle}{\left\langle}
\newcommand{\rrangle}{\right\rangle}

\newcommand{\bx}{\mathbf{x}}
\newcommand{\bX}{\mathbf{X}}
\newcommand{\bz}{\mathbf{z}}
\newcommand{\by}{\mathbf{y}}

\newcommand{\be}{\begin{equation*}}
\newcommand{\ee}{\end{equation*}}
\newcommand{\bea}{\begin{eqnarray*}}
\newcommand{\eea}{\end{eqnarray*}}
\newcommand{\ben}{\begin{eqnarray}}
\newcommand{\een}{\end{eqnarray}}
\newcommand{\beq}{\begin{equation}}
\newcommand{\eeq}{\end{equation}}
\newcommand{\enq}{\end{equation}}

\title{Gap Localization of TE-Modes by arbitrarily weak defects - multiband case}

\author{B.M.Brown}
\author{V.Hoang}
\author{M.Plum} 
\author{M.Radosz} 
\author{I.Wood} 

\begin{document}

\begin{abstract}
This paper considers the propagation of TE-modes in photonic crystal waveguides. The waveguide is created by introducing a linear defect into a periodic background medium. Both the periodic background problem and the perturbed problem are modelled by a divergence type equation. A feature of our analysis is that we allow discontinuities in the coefficients of the operator, which is required to model many photonic crystals. Using the Floquet-Bloch theory in negative order Sobolev spaces, we characterize the precise number of eigenvalues created by the line defect in terms of the band functions of the original periodic background medium for arbitrarily weak defects.
\end{abstract}

\maketitle

\section{Introduction}
Electromagnetic waves in periodically structured media, such as photonic crystals and metamaterials, are a subject of ongoing interest. Typically, the propagation of waves in such media exhibits \emph{band-gaps}; see e.g. \cite{Joann,KuchRev}. These are intervals on the frequency
or energy axis where propagation is forbidden. Mathematically, these correspond to gaps in the spectrum of the operator describing a problem with periodic background medium.
The existence of these gaps for certain choices of material coefficients was proved in \cite{CD99,FK96,HPW09} and in \cite{Fil03} for the full Maxwell case. 

In a previous paper \cite{BHPRW17}, we studied the propagation of TE-polarized waves in two-dimensional photonic crystals that contain line defects and gave rigorous sufficient conditions which imply spectral localization in band gaps. Our results were restricted to the case where only one band function (see \eqref{def_band_functions}) contributes to the edge of the band gap. In this paper, we deal with the general situation where multiple bands contribute to the edge of the gap. We also develop a new approach to characterize the precise number of eigenvalues created by the line defect in terms of the band functions of the original periodic structure. 

Our results are applicable to non-smooth coefficients. This is motivated by physical applications, where, to produce the typical band-gap spectrum,  the coefficient of the background medium is usually piecewise constant. See, for instance, \cite{CD99,FK96,Fil03}. In order to overcome the arising difficulties, we use  Floquet-Bloch theory in negative function spaces \cite{BHPW11}. Additionally, all our results do not depend on the precise geometry of the perturbation, e.g. the shape of the inclusions defined by the region within the periodicity cell where the 
perturbed material coefficients differ from the unperturbed ones. For a more detailed discussion of relevant background material, we refer to \cite{BHPRW17} and references therein. 

The structure of our paper is as follows: In  section \ref{sec_operators}  we give a brief description of the periodic problem and its perturbation by a line defect and formulate the operator-theoretic background. The following section \ref{sec:FB} introduces the Floquet-Bloch theory in negative spaces with the technical proof provided in Appendix \ref{FB}. Section \ref{sec_Rayleigh_Quotient} contains some key preparatory Lemmas and estimates. An upper estimate on the number of eigenvalues created in the band gap is given in section \ref{sec_upperbound} while section \ref{sec_LowerBound} provides a lower bound and combines all results to our main statement (Theorem \ref{theoremall}) on the precise number of eigenvalues.

We note that a variational method similar to the one here is used in \cite{Johnson} to prove generation of spectrum, though not the precise number of eigenvalues, in the band gaps of periodic Schr\"{o}dinger operators under a slightly weaker sign condition on the perturbation than we require here.

\section{The operator theoretic formulation}\label{sec_operators}
We  consider the propagation of electromagnetic waves in a non-magnetic, inhomogeneous medium
described by a varying dielectric function $\eps(\bX)$ with $\bX = (x, y, z)$. Assuming that
the magnetic field $\mathbf{H}$ has the form $\mathbf{H}= H(x, y) \hat{\bz}$, where $\hat{\bz}$ denotes the unit vector in the $z$-direction, we look for time-harmonic solutions to Maxwell's equations.  This leads to the equation
\begin{align} \label{eq:magfield}
- \nabla \cdot \frac{1}{\eps(\bx)} \nabla H = \lambda H
\end{align}
for the $\bz$-component $H$ of the magnetic field. Note that in the context of polarized waves, we
assume that all fields and constitutive functions depend only on $\bx = (x,y)$. 

The periodic background medium is characterised by $\eps_0(\bx)$, where for simplicity we assume that the unit square $[0, 1]^2$ is a cell of periodicity.

Let  $\hat{\bx}=(1,0)$ and $\hat \by=(0,1)$.   
We now introduce a line defect, which we assume  to be aligned along the $\hat{\bx}$-axis and preserving the periodicity in the $\hat{\bx}$-direction. In addition, the defect is assumed to be localised in the $\hat \by$-direction.
The new (and perturbed) system is described by a  dielectric function $\eps_1(\bx)$, periodic in $\hat \bx$-direction (see Figure \ref{figwav}), i.e.
\begin{align}
\eps_1(\bx + m \hat{\mathbf{x}}) = \eps_1(\bx) \quad (m\in \Z).
\end{align}
\begin{figure}
    \centering
    \includegraphics[scale=0.31]{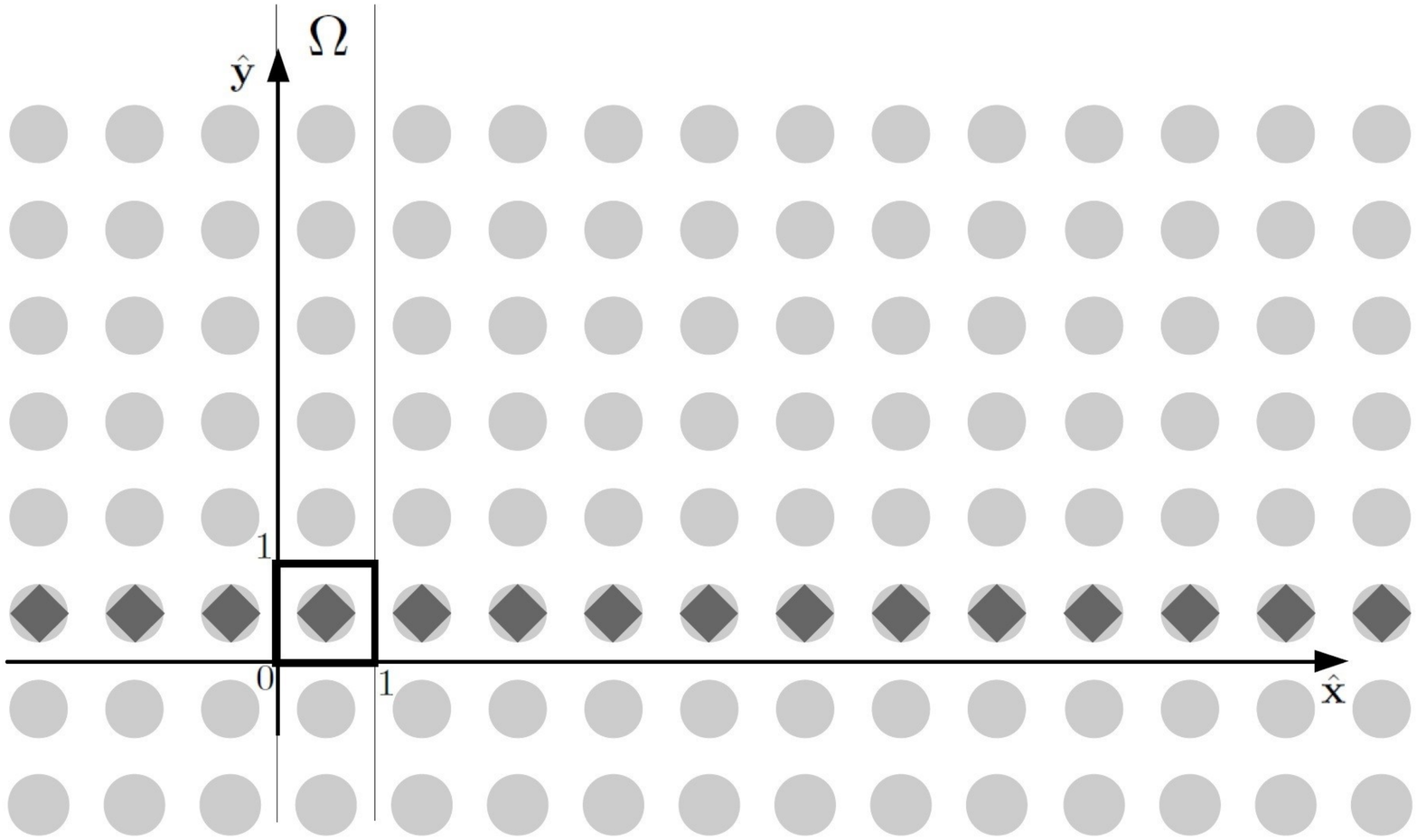}
    \caption{Illustration of the line defect and the strip $\Omega=(0,1)\times \R$.}
    \label{figwav}
\end{figure}

\begin{assumption}\label{asseps}
We make the following general assumptions on the material coefficients, valid throughout the paper:
\begin{enumerate}
\item[(i)]  $\eps_0,\eps_1\in L^\infty(\R^2)$.
\item[(ii)]
  $\eps_i\geq c_0>0$ for some constant $c_0$ and $i=0,1$.
\item[(iii)]The perturbation is nonnegative, i.e.~
\begin{align}\label{positivity}
\eps_1(\bx) - \eps_0(\bx) \geq 0.
\end{align}
\item[(iv)] There exists a ball $D$ such that $\eps_1-\eps_0>0$ on $D$. 
\end{enumerate}
\end{assumption}
Since both the perturbed and unperturbed systems are periodic in the $\hat{\bx}$-direction, we can apply Bloch's theorem \cite{OdehKeller,KuchmentBook}
to reduce both problems to problems on the strip $\Omega:=(0,1)\times\R$. 
For fixed quasi-momentum $k_x$ we introduce the space of quasi-periodic 
$H^1$-functions on $\Omega$
$$
H^1_{qp}(\Omega):=
\{ u \in H^1_\text{loc}(\R^2) :u\vert_{\Omega}\in H^1(\Omega) \hbox{ and } u(\bx+(m,0))=e^{i k_x m} u(\bx), m\in \Z, \bx\in \R^2 \}. 
$$

For $u,v\in H^1_{qp}(\Omega)$ consider the sesquilinear form 
\beq\label{Bo}
B_0 [u,v]=\int_\Omega \left(\frac{1}{\eps_0(\bx)}\nabla u \overline{\nabla v} + u\overline{v}\right)~d\bx.
\eeq

As $\eps_0$ is bounded and bounded away from zero, we can
introduce a new inner product on 
$H^1_{qp}(\Omega)$ given by $$ \la u,v\ra
_{H^1_{qp}(\Omega)}:=B_0[u,v]$$ which is equivalent to the standard
inner product in  $H^1(\Omega)$. When there is no danger of confusion, we denote the associated norm
$\norm{\cdot}_{H^1}$.
\begin{defn}
Let $H^{-1}_{qp}(\Omega)$ denote the dual space of $H^1_{qp}(\Omega)$. Let $\phi:H^1_{qp}(\Omega)\to H^{-1}_{qp}(\Omega)$ be defined by
\begin{equation}
(\phi u) [\varphi] =B_0[u,\varphi]  \mbox{  \; for\; all\; } u,\varphi\in H^1_{qp}(\Omega),
\label{eq:star}
\end{equation}
where the $w[\varphi]$-notation indicates the dual pairing, i.e. it is the action of the linear functional $w$
on the function $\overline \varphi$. 
\end{defn}
$\phi$ is an isometric isomorphism, and hence  the inner product on $H^{-1}_{qp}(\Omega)$ given by $$\llangle u,v\rrangle _{H^{-1}_{qp}(\Omega)}:=\llangle \phi^{-1}u,\phi^{-1}v\rrangle _{H^1_{qp}(\Omega)}$$ induces a norm which coincides  with the usual operator sup-norm on $H^{-1}_{qp}(\Omega)$.

After this preparation, we now introduce the realisations of the operators in $H^{-1}_{qp}(\Omega)$  and
  define the operator $\lo:D(\lo)\to H^{-1}_{qp}(\Omega)$ by
$D(\lo):=H^1_{qp}(\Omega)\subset H^{-1}_{qp}(\Omega)$ with
$$\lo u:=\phi u - u \label{eq:2.1}.$$
Then $\lo+1$ is bijective and  both $\lo$ and $\go:=(\lo+1)^{-1}$ are self-adjoint, see \cite[Proposition 4.1]{BHPRW17}.
$\lo$ corresponds to the fully periodic problem \eqref{eq:magfield} with $\eps=\eps_0$.

 The useful identity
\beq\label{eq:phi} \llangle u, v\rrangle _{H^{-1}} = \llangle \phi^{-1} u, \phi^{-1} v\rrangle _{H^1} = \llangle u, \phi^{-1} v\rrangle _{L^2}\quad \hbox{for} \quad u\in L^2(\Omega), v\in H^{-1}_{qp}(\Omega)\eeq
follows from the definitions of $\phi$ and $\lo $.

Let $(\Lambda_0,\Lambda_1)$ be a spectral gap for $\lo$ and
$\mu\in ((\Lambda_1+1)^{-1}, (\Lambda_0+1)^{-1})$. Then $1/\mu\in \rho(\lo+1)$, so
\beq\label{logo}\frac{1}{\mu}\left[\frac{1}{\mu}- (\lo +1)\right]^{-1}=((I-\mu (\lo +1))^{-1}=(I-\mu \go^{-1})^{-1}\eeq
is well-defined and maps $H^{-1}_{qp}(\Omega)$ bijectively onto $H^1_{qp}(\Omega)$. The operator
$(I-\mu \go^{-1})^{-1}$ is the solution operator to the problem 
$$ \llangle u,\varphi\rrangle_{L^2} -\mu \int_{\Omega} \left (\tfrac{1}{\eps_0} \nabla u \overline{\nabla \varphi} + u \overline\varphi\right)d\bx = f[\varphi], \quad \hbox{ for all } \varphi \in H^1_{qp}(\Omega)
$$
for a given $f\in H^{-1}_{qp}(\Omega)$.

We now examine the perturbed problem. Let the bilinear form $B_1$ and the operator $\lp:H^1_{qp}(\Omega)\to H^{-1}_{qp}(\Omega)$ be defined by
\begin{align}
\label{Gone} B_1[u,\varphi]:=\int_\Omega \left [  \dfrac{1}{\eps_1} \nabla u   \overline{\nabla \varphi} + u \overline\varphi \right ] dx \hbox{ and }
 ((\lp+1)u)[\varphi]=B_1(u,\varphi) \quad\hbox{ for  } u, \varphi \in H^1_{qp}(\Omega).
\end{align}
Moreover, we define $\gp=(\lp+1)^{-1}$. Then 
$\gp:H^{-1}_{qp}(\Omega)\to H^1_{qp}(\Omega)$ is a bounded non-negative operator (see \cite[Lemma 1 \& 2]{BHPRW17}.)

\begin{remark}
We note that just as in \cite[Section 5]{BHPW11}, the spectra of the $H^{-1}$-realizations $\lo$ and $\lp$ and the corresponding realizations of the operators in $L^2(\Omega)$ coincide. 
\end{remark}
Suppose now $(\Lambda_0, \Lambda_1)$ is a band gap of the unperturbed operator $\lo$.
We will give conditions which ensure that localized modes, i.e.~eigenvalues of the perturbed operator $\lp$, appear in the band gap
under arbitrarily weak perturbations  and  use a Birman-Schwinger-type reformulation to find the eigenvalues $\lambda$ of the operator $\lp$ in a spectral gap. For proofs of the results in this section and more details on the reformulation, see \cite[Section 5]{BHPRW17}.

Consider the operator
$$K:=(\go ^{-1}\gp -I):H^{-1}_{qp}(\Omega) \to H^{-1}_{qp}(\Omega),$$
set $\ck = \overline{\Ran K}\subseteq H^{-1}_{qp}(\Omega)$  and let  $P:H^{-1}_{qp}(\Omega)\to \ck$ be the orthogonal projection on $\ck$.
On $\ck$, we introduce a new inner product given by
\begin{align}\label{defInnerprod}
\la f, g \ra_\ck  := \la K f, g \ra_{H^{-1}}.
\end{align} 
The symmetry and definiteness of this inner product is shown in \cite[Appendix A]{BHPRW17}.

The following lemma gives useful estimates for $K$ in terms of the size of the perturbation. In particular it shows that for small perturbations, the only dependence of the bound for $\norm{K}$ on the perturbation $\eps_1$ is through the term $\norm{\frac{1}{\eps_0}-\frac{1}{\eps_1}}_\infty$. 
\begin{lemma} \label{lem12} The following estimates hold:
\begin{enumerate}
\item[(i)]$$\quad \norm{K}\leq  \norm{\gp  }_{H^{-1}\to H^1} \norm{\frac{1}{\eps_0}-\frac{1}{\eps_1}}_\infty$$  
\item[(ii)] 
$$\norm{ Ku}_{H^{-1}}^2\leq \norm{K}\norm{u}^2_\ck \quad (u\in\ck) $$
\item[(iii)]
Moreover, if $\eta:=\norm{\frac{1}{\eps_0}-\frac{1}{\eps_1}}_\infty<1/\norm{\go  }_{H^{-1}\to H^1}$, then
$$ 
\quad \norm{\gp  }_{H^{-1}\to H^1} \leq \frac{\norm{\go  }_{H^{-1}\to H^1}}{1-\eta \norm{\go }_{H^{-1}\to H^1}}.
$$
\end{enumerate}
\end{lemma}

\begin{proof}
See \cite[Lemma 5.2]{BHPRW17}.
\end{proof}

Next for $\mu=(\lambda+1)^{-1} \in ((\Lambda_1+1)^{-1}, (\Lambda_0+1)^{-1})$ define
\begin{align}\label{eq:Amu} 
A_\mu:= P(I-\mu \go^{-1})^{-1} K :\ck\to\ck. 
\end{align}

\begin{lemma}\label{lemmavar}
The equation $(\lp-\lambda) u=0$ with $\lambda\in(\Lambda_0,\Lambda_1)$
 has a non-trivial solution $u$ iff $-1$ is an eigenvalue of $A_\mu$, where $\mu=(\lambda+1)^{-1}$.
\end{lemma}

\begin{proof}
See \cite[Lemma 5.3]{BHPRW17}.
\end{proof}

To be able to use the variational characterisation of eigenvalues we need  the following properties of the operator $A_\mu$.
\begin{prop}  \label{AmuC}
$A_\mu:\ck\to\ck$ is a compact, symmetric operator on $\ck$.
\end{prop}

\begin{proof}
See \cite[Proposition 5.7]{BHPRW17}.
\end{proof}

\section{Floquet-Bloch theory in $H^{-1}$}\label{sec:FB}

For our results we will make use of Floquet-Bloch theory in $H^{-1}_{qp}(\Omega)$. We introduce the notation and state the results needed here. A fuller account with proofs of some properties of the Floquet-Bloch theory in $H^{-1}_{qp}$ can be found in \cite{BHPW11}. The Brillouin zone in our setting is
the interval $[-\pi,\pi]$. This corresponds to our periodic cell in $\hat \by$-direction which is the interval $[0,1]$.
\begin{defn} \label{defn1}
For all $k$ in the Brioullin zone $[-\pi,\pi] $, we introduce an extension operator
$E_k :L^2((0,1)^2 )\to L^2_{loc}(\Omega)$ with
$$(E_k  u)(x,y+p):=e^{ik p}u(x,y)$$
for all $(x,y)\in (0,1)^2 $, $p\in \Z$.

The partial Floquet transform
$$ U :L^2(\Omega)\to L^2((0,1)^2  \times  [-\pi,\pi] )$$
is defined on functions with compact support by
$$ (U  u)(x,y,k):=\frac{1}{ (2\pi)^{1/2}} \sum_{n \in \Z }e^{ik n} u(x,y-n)\quad \hbox{ for } (x,y)\in (0,1)^2 , k\in [-\pi,\pi] 
$$
and extended to $L^2(\Omega)$ by continuity.
\end{defn}

$U $ is an isometric isomorphism and
\beq\label{inverseFloquet}(U ^{-1} v)(x,y)=\frac{1}{ (2\pi)^{1/2}}\int_{-\pi}^\pi   (E_k  v(\cdot,\cdot, k)) (x,y) dk
\eeq 
(see \cite{KuchmentBook}).

\begin{defn} \label{defn2}
Let $H^1_{qp}((0,1)^2 )$ denote the set of restrictions of functions $u\in H^1_{qp}(\Omega)$ to $(0,1)^2$ endowed with the $H^1$-inner product.
For all $k \in [-\pi,\pi] ,$ let
$$\ch_k:=\{ u\in H^1_{qp}((0,1)^2 ): E_k  u \in H^1_{loc}(\Omega)\}. $$
Note that being an element of $\ch_k$ requires a weak form of semi-periodic boundary conditions on the boundary of $(0,1)^2 $.
We denote by $N_k$ the mapping
$$ N_k:\ch_0\to\ch_k, \quad (N_ku)(x,y):=e^{iky} u(x,y)$$
and extend it to a mapping  $\ch_0'\to\ch_k'$ between the dual spaces by
$$  
N_k u[\varphi]:= u[N_{k}^{-1}\varphi]  \hbox{ for all } u\in\ch_0', \varphi\in \ch_k.
$$
Let
\begin{align*}
\ch&=\left\{u\in L^2((0,1)^2  \times [-\pi,\pi] ): \forall' k \in [-\pi,\pi] \;\;\;u(\cdot,\cdot, k)\in \ch_k,\right.\\
&\quad\text{the mapping}\quad \left\{\begin{array}{l} [-\pi,\pi] \to\C\\ k\mapsto \langle N_k^{-1} u(\cdot,\cdot,k), \varphi\rangle_{H^1((0,1)^2 )}  \end{array}\right\} \text{is measurable for all } \varphi\in\ch_0,\\
&\;\;\;\left.\text{and}\;\;\norm{u}_{\ch}<\infty\right\}
\end{align*}
where, as usual, $\forall' k$ means for almost all $k$ and  the norm $\norm{\cdot}_{\ch}$ is induced by the inner product
$$ \langle u,v\rangle_\ch = \int_{-\pi}^\pi   \langle u(\cdot,\cdot,k),v(\cdot,\cdot,k)\rangle_{H^1((0,1)^2 )} dk.$$
$\ch$ can be viewed as the space of all functions $u(x,y,k)=(N_k
v(k))(x,y)$ with $v\in L^2([-\pi,\pi] ,\ch_0)$. 
By $\phi_k:\ch_k\to\ch'_k$  and $\phi_\ch:\ch\to\ch'$ we denote the canonical isometric isomorphisms (defined analogously to \eqref{eq:star}).
\end{defn}

\begin{remark}
$\mathcal{H}$ can also be defined as the direct integral of the $\mathcal{H}_k$, which are then regarded as fibers over $k\in [-\pi, \pi]$ (see e.g. \cite{RS}).
\end{remark}
Analogously to \eqref{eq:phi}, we get
\beq\label{eq:phik} \llangle u, v\rrangle _{\ch_k'} = \llangle \phi_k^{-1} u, \phi_k^{-1} v\rrangle _{\ch_k} = \llangle u, \phi_k^{-1} v\rrangle _{L^2}\quad \hbox{for} \quad u\in L^2((0,1)^2), v\in \ch_k'.\eeq

   Let $V$ be given by $V:=U\vert_{H^1_{qp}(\Omega)}$.
   For $u,v\in H^1_{qp}(\Omega)$ we have $V  u, V  v\in\ch$, and
   $$\int_{-\pi}^\pi  b_0 [V  u(\cdot,\cdot,k),V  v(\cdot,\cdot,k)]dk=B_0[u,v],$$
   where $b_0$ is defined as $B_0$ in \eqref{Bo} with the range of integration $\Omega$ replaced by $[0,1]^2$ (see \cite[Theorem 3.7]{BHPW11}). The form $b_0$ induces the inner product on the space $H_{qp}^1((0,1)^2)$ as well as on $\ch_k$ giving
   \ben \label{viso} \langle V  u, V  v\rangle_{\ch}= \langle  u,  v\rangle_{H^1(\Omega)}.\een
   Moreover, $V  :H^1_{qp}(\Omega)\to \ch$ is an isometric isomorphism (see \cite[Theorem 3.8]{BHPW11}), whence also its adjoint  $V^*:\ch'\to H^{-1}_{qp}(\Omega)$ is. In particular, $\ch$ is a Hilbert space. 
  The map $$\hat V  :=(V ^*)^{-1}:H^{-1}_{qp}(\Omega)\to \ch'$$ is an isometric isomorphism and $\hat V  \mid_{L^2(\Omega)}=U $ (see \cite[Lemma 3.9]{BHPW11}).
 For $k\in [-\pi,\pi] $, let $\ch_k$ be the domain of the operator $L_k$  defined in $\ch'_k$ by
    $$L_k:\ch_k\subseteq  \ch'_k\to \ch'_k,\;\;\;L_k u=\phi_k u -u.$$ 
    This means that (cf.~\eqref{eq:star})
    $$\langle (L_k+1) u,\varphi\rangle =b_0 [u,\varphi]\hbox{ for } u,\varphi\in\ch_k. $$
Note that $\ch_k$ is dense in $\ch_k'$ since $\ch_k$ is dense in $L^2((0,1)^2 )$ and thus, by duality $L^2((0,1)^2 )$ is dense in $\ch_k'$.
Analogously to the case of $\lo$, the operator $L_k$ is self-adjoint. $(L_k+1)^{-1}$ is compact since it is bounded from $\ch_k'$ to $\ch_k$, which is compactly embedded in $\ch_k'$. 

It is possible to transform the spectral problem for the operators $L_k$ which have $k$-dependent domains  to a spectral problem for an operator family where the $k$-dependence is transferred to the differential expression (see, e.g.~\cite{BHPW15}, for the transformation in a similar situation). This family is analytic of type (A) in the sense of Kato and using \cite[Theorem VII.3.9 and Remark VII.3.10]{Kato}, we can obtain sequences of real-valued functions $\{\lambda_s(k)\}_{s\in \N}$ and eigenfunctions $\{\varphi_s(k)\}_{s\in \N}$, normalized in $\ch_k'$. 
The functions  $\lambda_s(k)$ and  $\varphi_s(k)$ are all real-analytic functions in the variable $k$ on $[-\pi,\pi]$ and are such that 
\begin{align}\label{def_band_functions}
    (L_k+1) \varphi_s(k)=\lambda_s(k)\varphi_s(k).
\end{align}
We note that the eigenvalues are not necessarily ordered by magnitude. 
We call the functions $\lambda_s(k)$ the band functions and $\varphi_s(k)$  the Bloch functions.

Throughout, we will need to make the following non-degeneracy assumption on the band functions:
\begin{assumption}\label{asnc}
The band functions $\lambda_s$ are not constant as functions of  $k\in[-\pi,\pi]$.
\end{assumption}

For notational convenience, we also introduce 
$$\psi_s(\cdot,k):=\frac{1}{\sqrt{\lambda_s(k)+1}}\ \varphi_s(\cdot,k).$$
The set $\{\psi_s(\cdot,k) \}$  forms an orthonormal set in $L^2((0,1)^2)$, which is also complete as the set of eigenfunctions of the self-adjoint realisation of the operators in $L^2((0,1)^2)$. As a general rule, we will always extend the $\psi_s(k), \varphi_s(k)$ to the whole of $\Omega$ in a $k$-quasiperiodic manner, i.e.
$$
\psi_s(\cdot + m \hat{\by}, k) = e^{i k m}\psi_s(\cdot, k).
$$
In what follows,
for $f\in H^{-1}_{qp}(\Omega)$ we denote by $(\hat{V}f)_k$  the element of
$\ch_k'$,  defined by \beq\label{notation}
[ (\hat{V}f)_k][\varphi] :=\langle
(\phi_\ch^{-1} \hat V  f)(\cdot,k),\varphi\rangle
_{\ch_k}\hbox{ for } \varphi \in \ch_k.
\eeq

\begin{lemma}\label{lem:notation}
For almost all $k\in[-\pi,\pi]$ and $f\in H^{-1}_{qp}(\Omega)$
 \begin{equation}\label{Vhat}
 \phi_k^{-1}(\hat{V}f)_k=(\phi_\ch^{-1} \hat V  f)(\cdot,k).
 \end{equation}
\end{lemma}

\begin{proof}
Let $w\in\ch_k$. Then
\begin{equation*}
    \langle \phi_k^{-1}(\hat{V}f)_k,w\rangle_{\ch_k} = (\hat{V}f)_k[w]= \langle (\phi_\ch^{-1} \hat V  f)(\cdot,k),w\rangle_{\ch_k},
\end{equation*}
which proves the identity.
\end{proof}
Having introduced the required notation, we are now able to state the results on expansions of functions in terms of the Bloch waves needed for this paper. The proofs can be found in Appendix \ref{FB}.

\begin{prop}\label{prop:FB} 
\begin{enumerate}
	\item\label{spectra} $\sigma(\lo)=\overline{  \cup_k  \sigma(L_k)}$.
    \item\label{l1} For $f \in H^{-1}_{qp}(\Omega)$ and $ \lambda \not \in  \sigma(\lo)$,  
    \beq
    (V(\lo-\lambda)^{-1} f ) (x,k)=(L_k-\lambda)^{-1} [(\hat{V}f)_k] (x) \label{*}.
    \enq
holds.
\item \label{resexpansionlemma} For   $g \in H^{-1}_{qp}(\Omega)$ and $ \lambda \not \in  \sigma(\lo)$ the equality
$$ (\lo-\lambda)^{-1}g = \frac1{\sqrt{2\pi}}\sum_{s=1}^\infty \int_{-\pi}^\pi \dfrac1{\lambda_s(k)-\lambda} ( \hat V g)_ k[\psi_s(\cdot, k)] \psi_s( :,k) dk$$
holds, where the series converges in $L^2(\Omega)$.
\item \label{RQ} For $\mu \in  ((\Lambda_1+1)^{-1}, (\Lambda_0+1)^{-1})$ and $u\in H^{-1}_{qp}(\Omega)$, 
\bea &&
\llangle -\dfrac1\mu \left(\lo+1-\dfrac1\mu\right)^{-1} (K u), Ku\rrangle_{H^{-1}} \\
  &&=\int_{-\pi}^\pi  \sum_{s=1}^\infty\dfrac1{1-\mu(\lambda_s(k)+1)}     \dfrac1{\lambda_s(k)+1} \left|  \llangle\psi_s(k),\phi^{-1}_k (\hat V K u)_k\rrangle_{H^1(0,1)^2}\right|^2 dk.
  \eea
	\item \label{expansionlemma}
For $f\in H^{-1}_{qp}(\Omega)$, 
\beq
\norm{f}_{H^{-1}}^2 = \frac{1}{\sqrt{2\pi}} \sum_{s=1}^\infty\int_{-\pi}^\pi \frac{1}{\lambda_s(k)+1} \left| \la \phi_k^{-1}(\hat{V} f)_k, \psi_s(k)\ra_{H^1((0,1)^2)}  \right|^2~dk.
\eeq
\end{enumerate}
\end{prop}
We refer to Figure \ref{fig1} for an overview of the spaces and mappings discussed here.

\begin{center}
\begin{figure}[h]
\begin{tikzcd}
 H^1_{qp}(\Omega) \arrow[r, hook] \arrow[rr, bend left, "\phi=\mathfrak{L}_0+1"] \arrow[to=5-1, "V"] & L^2(\Omega) \arrow[r, hook] & H_{qp}^{-1}(\Omega) \arrow[ll, bend left, "\phi^{-1}=(\mathfrak{L}_0+1)^{-1}"] \arrow[to=5-3, xshift = 1.5ex, "\hat{V}"] \arrow[loop right, "K"]\\
& & &\\
& & &\\
& & &\\
\mathcal{H} \arrow[r, hook] \arrow[rr, bend left, "\phi_{\mathcal{H}}"] & L^2((0,1)^2) \arrow[r, hook] & \mathcal{H}' \arrow[ll, bend left, "\phi^{-1}_{\mathcal{H}}"] \arrow[to=1-3, "V^{*}"]\\
& & &\\
& & &\\
& & &\\
 \mathcal{H}_k \arrow[r, hook] \arrow[rr, bend left, "\phi_k=L_k+1"] \arrow[to=5-1, "\displaystyle\int^\bigoplus_{[-\pi, \pi]}~dk"] & L^2((0,1)^2) \arrow[r, hook] & \mathcal{H}_{k}' \arrow[ll, bend left, "\phi^{-1}_k"] \arrow[to=5-3, "~~~\displaystyle\int^\bigoplus_{[-\pi, \pi]}~dk"']
\end{tikzcd}
\label{fig1}\caption{Spaces, isomorphisms and key mappings. The symbol $\int^\bigoplus_{[-\pi, \pi]}~dk$ indicates the forming of a direct integral of the $\mathcal{H}_k$.}
\end{figure}
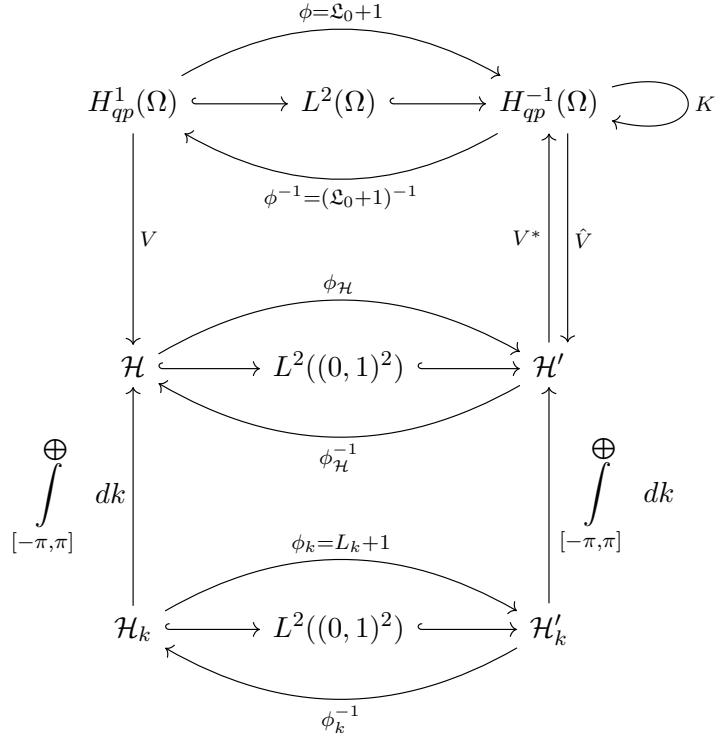
\end{center}

\section{Preparatory results}\label{sec_Rayleigh_Quotient}

Our strategy consists in following $\kappa_m(\mu)$, the $m$-th lowest negative eigenvalue (if it exists) of the
operator $A_\mu$, introduced in \eqref{eq:Amu}, as $\mu$ varies. 
The following standard variational characterisations (see, for example, \cite{Heuser}) hold:
\beq\label{defkappan}
\kappa_m(\mu) = \max_{\codim L = m-1}\ \inf_{\phi\in L} \frac{\la \phi, A_\mu \phi \ra_\ck }{\la\phi, \phi \ra_\ck } = \min_{\dim L =m} \ \max_{\phi\in L} \frac{\la \phi, A_\mu \phi \ra_\ck }{\la\phi, \phi \ra_\ck }.
\eeq
\begin{lemma}\label{lem:cont} For $\mu$ in the spectral gap $\left((\Lambda_1+1)^{-1},(\Lambda_0+1)^{-1}\right)$, the mapping
$\mu\mapsto \kappa_m(\mu)$ is continuous and increasing.
\end{lemma}
The proof is virtually identical to that of Lemma 6.1 in \cite{BHPW15}. 
We remind the reader that $\Lambda_1$ is the lowest point of a spectral band and lies at the top edge of a gap. The solutions of the equation $\lambda_s(k)=\Lambda_1$ will play an important role in our analysis. We first introduce the following sets:
\bea
&\Sigma=\{ (s,{ k}) \in\N\times[-\pi,\pi]: \lambda_{s}(k)=\Lambda_1\},\\
&S_{ k}=\{s\in\N:(s,{ k})\in \Sigma\}, \\
&S=\{ s\in\N :\text{   there\;is\;a\;} { k}\;  \text{with \;} (s,{k})\in \Sigma \} = \bigcup_{ k} S_{k}.
\eea
We will next see that the set $\Sigma$ is finite. In the following we denote the elements of $\Sigma$ by $(s_j,k_j), j=1,...,n$ and set $\psi_j=\psi_{s_j}(k_j)$ and $\varphi_j=\varphi_{s_j}(k_j)$. 

\begin{lemma}\label{finite}
The set $\Sigma$ is a non-empty finite set. Moreover, $\lambda_s( k)\to \infty$ as $s\to\infty$, uniformly in $k\in [-\pi,\pi].$
\end{lemma}
\begin{proof}
We first note the coincidence of the spectra of the $L^2$ and $H^{-1}$ realisations (see \cite[Section 5]{BHPW11}), so it is enough to consider the $L^2$-realisation of the operator. The result then follows from \cite[Proposition 3.2 and its proof]{BHPW15}.
\end{proof}
\begin{cor} \label{cor:1}
There is an $s_0\in \N$ such that for all $s\geq s_0$ and for all $ k \in [-\pi,\pi]$ we have $\lambda_s(k)\geq \Lambda_1$, while for all $s< s_0$ and for all $ k \in [-\pi,\pi]$, $\lambda_s(k)\leq \Lambda_0$ holds.
\end{cor}
\begin{proof} The assertion follows from  continuity of the band functions, existence of the spectral gap and Lemma \ref{finite}. 
\end{proof}

\begin{lemma}\label{lem:5}
The set $\Sigma$ is  isolated in the sense that there is  $\eta >0$ such that for all $s\not \in S$, $| \lambda_s(k)-\Lambda_1 |\geq \eta$  for all  $k\in [-\pi,\pi]$.
\end{lemma}
\begin{proof}
The proof is the same as that of \cite[Lemma 3.7]{BHPW14}, noting that analyticity and non-constancy of the band function in the one-dimensional variable $k$ are sufficient 
to avoid Assumption 3.3 in 
\cite{BHPW14} in the proof.
\end{proof}

Noting that $\eps_0-\eps_1$ is compactly supported in $\Omega$, 
for  $\psi\in H^1_{loc}(\Omega)$, let $(\lo-\lp)\psi$ be the element of $H^{-1}_{qp}(\Omega)$ defined by 
\beq\label{def1}
[(\lo-\lp)\psi][\varphi] := \int_\Omega \left( \frac{1}{\eps_0}-\frac{1}{\eps_1}\right) \nabla \psi \nabla \overline{\varphi} \hbox{ for all } \varphi\in H^1_{qp}(\Omega).
\eeq
Moreover, letting $H^1_c(\Omega)$ denote the functions in $H^1_{qp}(\Omega)$ with compact support, we define for  $\psi\in H^1_{loc}(\Omega)$, the element $\lo\psi$  of $\left(H^1_c(\Omega)\right)'$ by
\beq\label{def2}
[\lo\psi][\varphi] := \int_\Omega  \frac{1}{\eps_0} \nabla \psi \nabla \overline{\varphi} \hbox{ for all } \varphi\in H^1_c(\Omega).
\eeq
Define
\begin{equation}\label{def:L}
    {\mathcal L} = \{ u\in \ck :  \forall j=1,...,n.\quad [(\lo-\lp)\psi_j][\gp u]=0\},
\end{equation} 
where the action is interpreted as in \eqref{def1}.
\begin{remark}\label{rem:3}
Observe that the action of $K u$ on any $\varphi \in H^1_{qp}$ can be written as
$$
Ku[\varphi]=\int_\Omega \left( \frac{1}{\eps_0}-\frac{1}{\eps_1}\right) \nabla \gp u \overline{\nabla \varphi}.
$$
Since $\left( \frac{1}{\eps_0}-\frac{1}{\eps_1}\right)$ has compact support, the action of $K u$ can be extended to any $H^1_{\text{loc}}$-function $\varphi$. Hence we shall define
$$
K u[\varphi]:=\int_\Omega \left( \frac{1}{\eps_0}-\frac{1}{\eps_1}\right) \nabla \gp u \overline{\nabla\varphi} \quad (\varphi \in H_{loc}^1).$$
Then recalling \eqref{def1} we get
\beq\label{eq:Ku}
Ku[\psi_j]=\int \left( \frac{1}{\eps_0}-\frac{1}{\eps_1}\right) \nabla \gp u \overline{\nabla \psi_j}=[(\lo-\lp)\psi_j][\gp u].
\eeq
\end{remark}

\begin{lemma}\label{codim}
The codimension of $\mathcal L$ satisfies $\codim  {\mathcal L}  =n$.
\end{lemma}

\begin{proof}
For $u\in\ck$,
\bea
[(\lo-\lp)\psi_j][\gp u] &=& \int_\Omega \left( \frac{1}{\eps_0}-\frac{1}{\eps_1}\right) \nabla \psi_j \overline{\nabla \gp u}. 
\eea
Let $\theta\in C^\infty(\Omega)$ be compactly supported in the $\bf{\hat{y}}$-direction with $\theta=1$ on $[0,1]^2$. Then 
\ben \label{eq:theta}
\overline{[(\lo-\lp)\psi_j][\gp u]} &=& \int_\Omega \left( \frac{1}{\eps_0}-\frac{1}{\eps_1}\right) \overline{\nabla (\theta\psi_j)} \nabla \gp u\\ \nonumber
&=&
K u[\theta\psi_j]   = \llangle   K  u, \phi(\theta\psi_j) \rrangle _{H^{-1}} \\ \nonumber
 &=& \llangle   K u,P \phi(\theta\psi_j) \rrangle _{H^{-1}} 
 = \llangle    u,P \phi(\theta\psi_j) \rrangle _{\ck}.
\een
Hence, ${\mathcal L }= \Span \{ P\phi(\theta \psi_j): j=1,\dots, n \}^\perp$ and we need to show that 
$$\dim ( \Span \{ P\phi(\theta \psi_j): j=1,\dots, n \})=n.$$

Assume $\sum\alpha_j P\phi (\theta\psi_j)=0$ in  $\ck$.
As $K$ is symmetric and non-negative in $H^{-1}_{qp}(\Omega)$, this is equivalent to $\Psi=\sum \alpha_j \phi (\theta\psi_j) \in \ker K$.
Now, $K\Psi=0$ is equivalent to $\go \Psi =\gp\Psi$.
Let $v:=\go \Psi$.  Then $(\lo+1) v = (\lp +1) v=\Psi$,  so 
$$0=((\lo-\lp )v)[v] = \int \left(\frac{1}{\eps_0}-\frac{1}{\eps_1} \right) | \nabla v|^2$$ and thus $\nabla v|_D=0$.
Hence $\lo v|_D=\lp  v|_D=0$ (in the sense that for any $\varphi\in C_c^\infty(D)$ we have
$[\lo v][\varphi]= [\lp v][\varphi]=0$).
Moreover, $v=\phi^{-1}\Psi=\sum\alpha_j \theta\psi_j$. Therefore, for any $\varphi\in C_c^\infty(D)$,
\bea
[\lo v][\varphi]&=& \sum\alpha_j \int_\Omega  \frac{1}{\eps_0} \nabla   (\theta\psi_j)\nabla \overline{\varphi} \\
&=& \sum\alpha_j \int_D  \frac{1}{\eps_0} \nabla  ( \theta \psi_j)\nabla \overline{\varphi} \\
&=& \sum\alpha_j\int_D  \frac{1}{\eps_0} \nabla   \psi_j\nabla \overline{\varphi} \\
&=& \sum\alpha_j \int_\Omega  \frac{1}{\eps_0} \nabla \psi_j\nabla \overline{\varphi} \\
&=&   \sum\alpha_j \Lambda_1 \psi_j[ \varphi] \ =\ \Lambda_1 v [\varphi].
\eea
So $\lo v\vert_D=\Lambda_1 v\vert_D$,  and hence $v|_D=0$.  By unique continuation, see \cite{Ale12}, $v=0$ and as the $\psi_j$ are linearly independent, we get  $\alpha_j=0$ for all $j$. 
\end{proof}

\section{Upper bound on the number of eigenvalues} \label{sec_upperbound}

The main result in this section will require the following additional non-degeneracy assumption on the band functions $\lambda_s(k)$.
\begin{assumption}\label{as1}
There are $\alpha>0$   and $\delta>0$ such that   for all $(\hat s, \hat k) \in \Sigma$ and ${ k }\in [-\pi,\pi] $ satisfying  $|{  k} -{ \hat k}| \leq \delta$, 
$$\lambda_{\hat s} ({ k})\geq \Lambda_1+\alpha | {k}-{ \hat k}|^2$$
holds.
\end{assumption}

\begin{remark}\label{rem1}
The assumption is true if  the zero of $\lambda_{\hat s} (\hat{{ k}}) -\Lambda_1$ is only of order $2$. Non-degeneracy assumptions of a similiar form are common in the mathematical and
physical literature (see e.g.~\cite{KR12} and references therein) and
are believed to be ``generically" true. In other words,  it is believed that degeneracy of the band function can be removed by a small perturbation of the coefficients of the
differential operator.
\end{remark}

The next lemma provides a uniform bound on contributions to the Rayleigh quotient away from points in $\Sigma$.

\begin{lemma}\label{lem:unif}
Let $\mu\in ((\Lambda_1+1)^{-1}, (\Lambda_0+1)^{-1})$.
If $\left| 1 -\mu (\lambda_s(k)+1)\right|^{-1}$ is uniformly bounded for $(s,k)$ in a set $\widetilde S\times J\subseteq \N\times[-\pi,\pi]$, then 
$$ \sum_{s\in \widetilde S} \int_{J} \dfrac1{ 1 -\mu (\lambda_s(k)+1)}\dfrac1{\lambda_s(k)+1} \left|  \llangle\psi_s(k),\phi^{-1}_k (\hat V K u)_k\rrangle_{H^1(0,1)^2}\right|^2 dk\geq -C_\mu\norm{\frac{1}{\eps_0}-\frac{1}{\eps_1}}_\infty \norm{u}^2_{\ck} .$$
\end{lemma}

\begin{proof}
Note that the order of integration over $J$ and summation over $s$ can be exchanged by the monotone convergence theorem. We have 
\bea
&&\int_J \sum_{s\in \widetilde S} \dfrac1{1-\mu(\lambda_s(k)+1)}     \dfrac1{\lambda_s(k)+1} \left|  \llangle\psi_s(k),\phi^{-1}_k (\hat V K u)_k\rrangle_{H^1(0,1)^2}\right|^2 dk\\
&& \geq -C_\mu \int_J \sum_{s\in \widetilde S}   \dfrac1{\lambda_s(k)+1} \left|  \llangle\psi_s(k),\phi^{-1}_k (\hat V K u)_k\rrangle_{H^1(0,1)^2}\right|^2 dk\\
&&\geq -C_\mu \int_{-\pi}^\pi \sum_{s}   \dfrac1{\lambda_s(k)+1} \left|  \llangle\psi_s(k),\phi^{-1}_k (\hat V K u)_k\rrangle_{H^1(0,1)^2}\right|^2 dk\\
&&= -C_\mu \norm{Ku}_{H^{-1}}^2\ \geq\ -C_\mu \norm{\frac{1}{\eps_0}-\frac{1}{\eps_1}}_\infty \norm{u}_{\ck}^2,
\eea
 where the equality follows from  Proposition \ref{prop:FB} \eqref{expansionlemma} and the final inequality from Lemma \ref{lem12}.
\end{proof}

Before stating the first main result we introduce an auxilliary function $f$, which will play a crucial role in the estimates of the Rayleigh quotient, and prove some identities and estimates involving $f$.

For $\tilde{k}$ such that $k_j+\tilde{k}\in[-\pi,\pi]$  ($j=1,...,n$) and $u\in H^{-1}_{qp}(\Omega)$ let
\beq\label{def:f} 
f(\tilde k, u):=\sum_{j=1}^n\sum_{s\in S_{k_j} }  \left|\la (\hat{V} Ku)_{k_j+\tilde{k}},  \varphi_s(\cdot, k_j+\tilde k)\ra_{H^{-1}([0,1]^2)} \right|^2.
\eeq
\begin{lemma}
The function $f$ from \eqref{def:f} can be represented as follows:
\ben\label{fversions1}
f(\tilde k, u)
&=&
\sum_{j=1}^n
\sum_{s\in S_{k_j} } 
\dfrac1{(\lambda_{s} (k _j+\tilde k)+1)^2}|(\hat{V} Ku)_{k_j+\tilde{k}}[\varphi_{s}(\cdot, k_j+\tilde k)]|^2 \\
&=&\sum_{j=1}^n
\sum_{s\in S_{k_j} } 
\dfrac1{(\lambda_{s} (k _j+\tilde k)+1)^2} \left|  \llangle \phi^{-1}_{k_j+\tilde{k}}(\hat{V} Ku)_{k_j+\tilde{k}}, \varphi_{s}(\cdot, k_j+\tilde k)\rrangle _{H^1([0,1]^2)}\right|^2. \label{fversions2}
\label{fversions3}\\
f(\tilde k, u) &=& \frac{1}{\sqrt{2\pi}}\sum_{j=1}^n
\sum_{s\in S_{k_j} } 
\dfrac1{(\lambda_{s} (k _j+\tilde k)+1)^2}\left|Ku [\varphi_{s}(\cdot,k_j+\tilde k)]\right|^2,\label{fversions4}
\een
where the action is considered as an $H^{-1}((0,1)^2)-H^{1}((0,1)^2)$-pairing. Moreover, 
\beq\label{eq:ian}
f(\tilde k, u)= \frac{1}{\sqrt{2\pi}}\sum_{j=1}^n
\sum_{s\in S_{k_j} }
 \dfrac1{(\lambda_{s} (k _j+\tilde k)+1)^2}\left|[(\lo-\lp)\varphi_{s}(\cdot,k_j+\tilde k)][\gp u]\right|^2.
\eeq 
\end{lemma}

\begin{proof}
We have
\begin{align*}
(\hat{V} Ku)_{k_j+\tilde{k}} [\varphi_{s}(k_j+\tilde{k})]   
 {}& =  \llangle   (\hat{V} Ku)_{k_j+\tilde{k}},\phi_{k_j+\tilde{k}} \varphi_{s}(k_j+\tilde{k})\rrangle _{H^{-1}([0,1]^2)}  \\
	{}& = (\lambda_{s}(k_j+\tilde{k})+1) \llangle   (\hat{V} Ku)_{k_j+\tilde{k}}, \varphi_{s} (k_j+\tilde{k})\rrangle_{H^{-1}([0,1]^2)},
  \end{align*}
which proves \eqref{fversions1}.

From \eqref{notation} and Lemma \ref{lem:notation}  it follows that 
$$(\hat{V} Ku)_{k_j+\tilde{k}}[\varphi_{s}(\cdot, k_j+\tilde k)] =  \llangle \phi^{-1}_{k_j+\tilde{k}}(\hat{V} Ku)_{k_j+\tilde{k}}, \varphi_{s}(\cdot, k_j+\tilde k)\rrangle _{H^1([0,1]^2)},$$
so \eqref{fversions2} holds.

We next prove \eqref{fversions4}.
In order to make use of the explicit form of the Floquet transform on compactly supported functions, we let $\Theta_N$ be  a cut-off function with $\Theta_N(y)=\Theta_1(y/N)$ and $\Theta_1\in\C_c^\infty(\R)$ with $\Theta_1(y)=1$ for $|y|\leq 1$ and $\Theta_1(y)=0$ for $|y|\geq 2$.
Applying the Floquet transform $V$ in $H_{qp}^1(\Omega)$ to the function $ \Theta_N (\lo+1)^{-1} Ku$ we get
\bea 
&&\llangle V (\Theta_N (\lo+1)^{-1} Ku)(\cdot, k_j+\tilde{k}),\varphi \rrangle_{H^1((0,1)^2)}\\ &&=\dfrac1{\sqrt{2\pi}} \sum_{p\in  \Z} e^{ i (k_j+\tilde{k})  p} \llangle \Theta_N (\lo+1)^{-1} K u\left(\cdot-\twovec{0}{p}\right),\varphi\rrangle_{H^1((0,1)^2)}\\
&&=\dfrac1{\sqrt{2\pi}} \sum_{p\in  \Z} e^{ i (k_j+\tilde{k})  p}\llangle \Theta_N(\lo+1)^{-1} K u,\varphi \left(\cdot+\twovec{0}{p}\right)\rrangle_{H^1((0,1) \times  (-p, -p+1))}\\
&&=\dfrac1{\sqrt{2\pi}} \sum_{p\in  \Z} \llangle \Theta_N(\lo+1)^{-1} K u,\varphi \rrangle_{H^1((0,1) \times  (-p, -p+1))}\\
&&=\dfrac1{\sqrt{2\pi}}  \llangle \Theta_N(\lo+1)^{-1} K u,\varphi \rrangle_{H^1(\Omega)}.
\eea

We now argue that in the limit $N\to\infty$, we can move $\Theta_N$ to the other side of the inner product.
Observe that
$$\llangle \Theta_N (\lo+1)^{-1} K u, \varphi_s \rrangle_{H^1(\Omega)}=\llangle \Theta_N (\lo+1)^{-1} K u, \varphi_s \rrangle_{L^2}+\llangle \eps_0^{-1} \nabla(\Theta_N (\lo+1)^{-1} K u),\nabla \varphi_s \rrangle_{L^2}.$$
Clearly, the first term allows moving  $\Theta_N$ to the right and it remains to show that 
$$\lim_{N\to\infty} \llangle \eps_0^{-1}\nabla(\Theta_N (\lo+1)^{-1} K u),\nabla \varphi_s \rrangle_{L^2} = \lim_{N\to\infty} \llangle \eps_0^{-1}\nabla (\lo+1)^{-1} K u,\nabla (\Theta_N \varphi_s) \rrangle_{L^2}.$$
Therefore, it suffices to show that 
$$\lim_{N\to\infty} \llangle \eps_0^{-1} \nabla(\Theta_N) (\lo+1)^{-1} K u,\nabla \varphi_s \rrangle_{L^2} = \lim_{N\to\infty} \llangle \eps_0^{-1}\nabla (\lo+1)^{-1} K u,\nabla (\Theta_N) \varphi_s \rrangle_{L^2}$$
and we will see that both limits vanish. Now, 
       \bea  \llangle \eps_0^{-1}(\lo+1)^{-1}Ku, (\nabla \Theta_N)\nabla\varphi_s\rrangle_{L^2(\Omega)} \leq \norm{ \eps_0^{-1}(\lo+1)^{-1}Ku}_{L^2(\Omega)} \norm{ (\nabla \Theta_N)\nabla \varphi_s}_{L^2(\Omega)}  \\ 
       \leq \norm{\eps_0^{-1}(\lo+1)^{-1} Ku}_{L^2(\Omega)} \dfrac C {\sqrt N} \norm {\nabla \varphi_s}_{L^2((0,1)^2)},
       \eea
       as
       $$ \norm{(\nabla \Theta_N) \nabla \varphi_s}^2_{L^2(\Omega)}= \int_{\supp(\nabla\Theta_N)} | \nabla \Theta_N|^2 | \nabla \varphi_s|^2 \leq \dfrac C{N^2} N \norm{\nabla \varphi_s}^2_{L^2((0,1)^2)}.$$
       The other term can be estimated in a similar manner.
			
			 Using \eqref{notation}, Lemma \ref{lem:notation} and \eqref{*}, 
			\begin{eqnarray*}
			(\hat{V} Ku)_{k_j+\tilde{k}}[\varphi] &=& \langle (\phi_\ch^{-1} \hat V  f)(\cdot,k),\varphi\rangle_{H^1} \\
			&=& \langle(\phi_{k_j+\tilde{k}}^{-1}  (\hat{V} Ku)_{k_j+\tilde{k}},\varphi\rangle_{H^1}\\
			&=& \llangle  ( V (\lo+1)^{-1} Ku)(\cdot,k_j+\tilde{k}),\varphi\rrangle_{H^1},
			\end{eqnarray*}
			which implies that
			\bea
			(\hat{V} Ku)_{k_j+\tilde{k}}[\varphi] &=& \frac{1}{\sqrt{2\pi}} \lim_{N\to\infty} \llangle (\lo+1)^{-1} K u),\Theta_N\varphi \rrangle_{H^1(\Omega)}\\
			&=& \frac{1}{\sqrt{2\pi}} \lim_{N\to\infty} Ku [\Theta_N\varphi_s]\  = \ \frac{1}{\sqrt{2\pi}} Ku [\varphi_s],
			\eea
			where the last equality follows from compactness of the support of $Ku$. Equation \eqref{fversions4} now follows from \eqref{fversions2}. To obtain \eqref{eq:ian}, we use Remark \ref{rem:3}.
\end{proof}

\begin{lemma}\label{flip} Let $\cL$ be the space defined in \eqref{def:L}.
For $u\in\cL$ the function
$f(\tilde k, u)$ satisfies the estimate
$$|f(\tilde{k},u)| \leq C |\tilde{k}|^2 \norm{\frac{1}{\eps_0}-\frac{1}{\eps_1}}_\infty \norm{u}_{\ck}^2.$$
\end{lemma}
\begin{proof}
We use \eqref{fversions3}. First note the following:
\bea
\dfrac{Ku [\varphi_s(\cdot,k_j+\tilde{k})]}{\lambda_s (k _j+\tilde k)+1} &=&\dfrac{\llangle Ku, (L_{k_j+\tilde{k}}+1)\varphi_s(\cdot,k_j+\tilde{k})\rrangle_{H^{-1}((0,1)^2)}}{(\lambda_s (k _j+\tilde k)+1)}
\ =\ \llangle Ku,\varphi_s(\cdot,k_j+\tilde{k})\rrangle_{H^{-1}((0,1)^2)}.
\eea

In particular, using \eqref{eq:Ku}, for $u\in\cL$ we obtain
\bea
\dfrac1{(\lambda_s (k _j+\tilde k)+1)} Ku [\varphi_s(\cdot,k_j+\tilde{k})]
&=&\llangle Ku,\varphi_s(\cdot,k_j+\tilde{k})-\varphi_s(\cdot,k_j)\rrangle_{H^{-1}((0,1)^2)}.
\eea
As the $\varphi_s$ depend analytically on $k$, 
$$\norm{\varphi_s(\cdot,k_j+\tilde{k})-\varphi_s(\cdot,k_j)}_{H^{-1}((0,1)^2)}\leq C|\tilde{k}|,$$
and we get for $u\in\cL$ that 
\bea
|f(\tilde{k},u)| &\leq& C |\tilde{k}|^2 \norm{Ku}_{H^{-1}((0,1)^2)}^2 \ \leq \ C |\tilde{k}|^2 \norm{\frac{1}{\eps_0}-\frac{1}{\eps_1}}_\infty \norm{u}_{\ck}^2,
\eea
completing the proof.
\end{proof}

\begin{lemma}\label{flowerbound}
 There exists $C>0$ such that  $f(0,u)\geq C\norm{u}_\ck^2$ for  all $u\in \cL^\perp = \Span \{ P\phi(\theta \psi_j): j=1\dots n \}$, where 
 $\theta\in C^\infty(\Omega)$ is any function  compactly supported in the $\bf{\hat{y}}$-direction with $\theta=1$ on $[0,1]^2$.
\end{lemma}

\begin{proof}
Using \eqref{eq:ian}  and \eqref{eq:theta}, we have 
\bea 
     f(0,u)&=& \frac{1}{\sqrt{2\pi}}\sum_{j=1}^n\sum _{s \in S_{k_j} } \dfrac1{(\Lambda_1+1)^2}\left|[(\lo-\lp)\varphi_{s}(\cdot,k_j)][\gp u]\right|^2 \\
				&=& \frac{1}{\sqrt{2\pi}}\sum_{j=1}^n \dfrac1{\Lambda_1+1}\left|\llangle u, P\phi(\theta\psi_j)\rrangle_\ck\right|^2.
\eea
Now let $u=  \sum_{\mu=1}^n\alpha_\mu  P\phi(\theta \Psi_\mu)$ and set $\theta_\mu=P\phi(\theta \Psi_\mu)$. Then
\bea 
		f(0,u) 
			&=& \frac{1}{\sqrt{2\pi}(\Lambda_1+1)}
			\sum_{j} \sum_{\mu,\nu}    \alpha_\mu\overline{\alpha_\nu} \llangle  \theta_\mu,  \theta_j\rrangle_\ck	
		\llangle \theta_j, \theta_\nu\rrangle_\ck\  = \ \frac{ \alpha^*G\alpha}{\sqrt{2\pi}(\Lambda_1+1)}\ \geq \ \frac{\lambda_{min}(G)\norm{\alpha}^2}{\sqrt{2\pi}(\Lambda_1+1)},
\eea
where $G$ is an $n\times n$-matrix with entries
\bea 
     G_{\mu,\nu}=  &\sum_{j}  \llangle  \theta_\mu,  \theta_j\rrangle_\ck	
		\llangle \theta_j, \theta_\nu\rrangle_\ck .
     \eea
       Then $G=\widetilde{G}^2$ where $\widetilde{G}_{\gamma, \beta} =\llangle \theta_\gamma ,  \theta_\beta \rrangle_\ck$.
By the proof of Lemma \ref{codim}, the set $\{\theta_j: j=1,...,n\}$ is linearly independent, so 
 $\widetilde{G}$ is a positive definite Hermitian matrix and also its square $G$ is.
  
Now,
$$\norm{u}^2_\ck=\la Ku,u \ra_{H^{-1}} = \sum_{i,j} \alpha_i\overline{\alpha_j}\la K \theta_i,\theta _j\ra_{H^{-1}} = \sum_{i,j} \alpha_i\overline{\alpha_j}\la  \theta_i,\theta _j\ra_{\ck} =  \alpha^* \widetilde{G}\alpha\ \leq\ \lambda_{max}(\widetilde{G})\norm{\alpha}^2.$$

Thus, 
$$f(0,u)\geq \frac{1}{\sqrt{2\pi}(\Lambda_1+1)} \frac{\lambda_{min}(G)}{\lambda_{max}(\widetilde{G})} \norm{u}^2_\ck.$$
\end{proof}

We now state the main result of this section.

\begin{theorem}\label{upper} 
Let Assumptions \ref{asseps}, \ref{asnc} and \ref{as1} hold. Then there exists $c>0$
such that if $\norm{\frac{1}{\eps_0}-\frac{1}{\eps_1}}_\infty<c$, then the operator $\lp $ has at most {$n=|\Sigma|$} eigenvalues in the spectral gap $(\Lambda_0,\Lambda_1)$ of the operator $\lo$.
\end{theorem}

\begin{proof}
We start by noting an equality for the Rayleigh quotient. Let $u\in \ck$. Then by using Proposition \ref{prop:FB} \eqref{RQ}, 
  \begin{eqnarray}\label{Rayleigh}
  \llangle A_\mu u,u\rrangle_{\ck} = \int_{-\pi}^\pi \sum_{s=1}^\infty\dfrac1{1-\mu(\lambda_s(k)+1)}     \dfrac1{\lambda_s(k)+1} \left|  \llangle\psi_s(k),\phi^{-1}_k (\hat V K u)_k\rrangle_{H^1(0,1)^2}\right|^2 dk.
  \end{eqnarray}
for $\mu\in ((\Lambda_1+1)^{-1}, (\Lambda_0+1)^{-1})$.
	
By continuity of the band function $\lambda_s$ we have, for each $s \in \N$, either   $\lambda_s({ k}) \leq \Lambda_0$  for all ${ k} \in  [-\pi,\pi] $
 or
$\lambda_s({ k})\geq \Lambda_1$ for all ${ k} \in  [-\pi,\pi] $. In the first case, 
$1/[(1-\mu(\lambda_s(k)+1))(\lambda_s(k)+1)]\geq 0$ while  in the second case, we  have  the reverse inequality.
Therefore, with $s_0$ as in Corollary \ref{cor:1}, 
\ben\label{rayleighlow}
\llangle A_\mu u,u\rrangle_{\ck} &\geq&
\int_{-\pi}^\pi \sum_{s\geq s_0} \dfrac1{1-\mu(\lambda_s(k)+1)}     \dfrac1{\lambda_s(k)+1} \left|  \llangle\psi_s(k),\phi^{-1}_k (\hat V K u)_k\rrangle_{H^1(0,1)^2}\right|^2 dk\\ \nonumber
&=& \int_{-\pi}^\pi \sum_{s\in S} \dfrac1{1-\mu(\lambda_s(k)+1)}     \dfrac1{\lambda_s(k)+1} \left|  \llangle\psi_s(k),\phi^{-1}_k (\hat V K u)_k\rrangle_{H^1(0,1)^2}\right|^2 dk\\ \nonumber
&& + \int_{-\pi}^\pi \sum_{s\geq s_0, s\not\in S} \dfrac1{1-\mu(\lambda_s(k)+1)}     \dfrac1{\lambda_s(k)+1} \left|  \llangle\psi_s(k),\phi^{-1}_k (\hat V K u)_k\rrangle_{H^1(0,1)^2}\right|^2 dk.
 \een

We first consider the second sum. By Lemma \ref{lem:5} and Lemma \ref{lem:unif}, it can be bounded below by
\ben\label{norm_u_curly_K}
 -C \norm{\frac{1}{\eps_0}-\frac{1}{\eps_1}}_\infty \norm{u}_{\ck}^2.
\een
 
Now, we turn our attention to the first sum. We remind the reader that the set $\Sigma$ consists of the elements $(s_j, k_j)$ with $j=1, \ldots, n$.
We split the domain of integration into balls of radius $\delta$ around the points ${ k}_j$ and the complement of the union of these balls in $[-\pi,\pi]$, where $\delta$ is chosen
as in Assumption \ref{as1}. 
Then

\bea
&& \int_{-\pi}^\pi \sum_{s\in S} \dfrac1{1-\mu(\lambda_s(k)+1)}     \dfrac1{\lambda_s(k)+1} \left|  \llangle\psi_s(k),\phi^{-1}_k (\hat V K u)_k\rrangle_{H^1(0,1)^2}\right|^2 dk \\
&=&   \sum_{s\in S} \left[\sum_{
\stackrel {j=1}{ s_j=s} }
^n
 \int_{B_\delta({ k}_j)}
 \dfrac1{1-\mu(\lambda_s(k)+1)}     \dfrac1{\lambda_s(k)+1} \left|  \llangle\psi_s(k),\phi^{-1}_k (\hat V K u)_k\rrangle_{H^1(0,1)^2}\right|^2 dk \right.   \\
&&  + \left. \int_{R_s}
 \dfrac1{1-\mu(\lambda_s(k)+1)}     \dfrac1{\lambda_s(k)+1} \left|  \llangle\psi_s(k),\phi^{-1}_k (\hat V K u)_k\rrangle_{H^1(0,1)^2}\right|^2 dk
   \right]
\eea
where  $R_s:=[-\pi,\pi] \backslash \cup_{\stackrel{j=1}{ s_j=s}}^{\ n} B_\delta({ k}_j) $. On $R_s$  we again use that $(1-\mu(\lambda_s(k)+1))^{-1}  $
is uniformly bounded (with respect to $s$ and ${k}$), since the continuous function $\lambda_s(\cdot)-\Lambda_1$  is positive and therefore positively bounded away from $0$ on the compact set $R_s$. Using Lemma \ref{lem:unif} again, the sum of the last integrals can be bounded below by \eqref{norm_u_curly_K}.

 It remains to estimate the sum of the integrals over $B_\delta({ k}_j)$. Exchanging the order of the sums  which can only add negative terms (if $s\in S_{ { k}_j} $  for several $j$) and then shifting the integration variable yields
\bea
&& \sum_{s\in S} \sum_{
\stackrel {j=1}{ s_j=s} }
^n
 \int_{B_\delta({ k}_j)}
 \dfrac1{1-\mu(\lambda_s(k)+1)}     \dfrac1{\lambda_s(k)+1} \left|  \llangle\psi_s(k),\phi^{-1}_k (\hat V K u)_k\rrangle_{H^1(0,1)^2}\right|^2 dk \\
 &\geq&  \sum_{j=1}^n   \sum_{s\in S_{{ k}_j} }
 \int_{B_\delta({ k}_j)}
 \dfrac1{1-\mu(\lambda_s(k)+1)}     \dfrac1{\lambda_s(k)+1} \left|  \llangle\psi_s(k),\phi^{-1}_k (\hat V K u)_k\rrangle_{H^1(0,1)^2}\right|^2 dk \\
 &=&    \sum_{j=1}^n   \sum_{s\in S_{{ k}_j} }
  \int_{B_\delta(0)}
 \dfrac1{1-\mu(\lambda_s({ k}_j+\widetilde { k})+1)}     \dfrac1{\lambda_s({ k}_j+\widetilde { k})+1} \left|  \llangle\psi_s({ k}_j+\widetilde { k}),\phi^{-1}_{{ k}_j+\widetilde { k}} (\hat V K u({ k}_j+\widetilde { k}))\rrangle_{H^1(0,1)^2}\right|^2 d\widetilde { k}\\
  &=&    \sum_{j=1}^n   \sum_{s\in S_{{ k}_j} }
  \int_{B_\delta(0)}
 \dfrac1{1-\mu(\lambda_s({ k}_j+\widetilde { k})+1)}     \dfrac1{(\lambda_s({ k}_j+\widetilde { k})+1)^2} \left|  \llangle\varphi_s({ k}_j+\widetilde { k}),\phi^{-1}_{{ k}_j+\widetilde { k}} (\hat V K u({ k}_j+\widetilde { k}))\rrangle_{H^1(0,1)^2}\right|^2 d\widetilde { k}\\
 &\geq& \int_{B_\delta(0)} \frac{1}{1-\mu(\Lambda_1+\alpha \tilde{k}^{2}+1)} f(\tilde{k},u) d\tilde{k},
\eea
where in the last step we have used Assumption \ref{as1} and Equation \eqref{fversions2}.
Now, for $u\in\cL$, by Lemma \ref{flip}, 
\bea
 \int_{B_\delta(0)} \frac{1}{1-\mu(\Lambda_1+\alpha \tilde{k}^{2}+1)} f(\tilde{k},u) d\tilde{k}
 &\geq& C \int_{B_\delta(0)} \frac{|\tilde{k}|^2 }{1-\mu(\Lambda_1+\alpha \tilde{k}^{2}+1)} d\tilde{k}\ \norm{\frac{1}{\eps_0}-\frac{1}{\eps_1}}_\infty \norm{u}_{\ck}^2\\
 &\geq& -\tilde{C} \norm{\frac{1}{\eps_0}-\frac{1}{\eps_1}}_\infty \norm{u}_{\ck}^2.
\eea

Combining all our results, we get that for $u\in\cL$ 
\bea
\llangle A_\mu u,u\rrangle_{\ck}\geq -C \norm{\frac{1}{\eps_0}-\frac{1}{\eps_1}}_\infty \norm{u}_{\ck}^2
\eea
for some $C>0$, independent of $\mu \in  ((\Lambda_1+1)^{-1}, (\Lambda_0+1)^{-1})$.
Therefore, if $C\norm{\frac{1}{\eps_0}-\frac{1}{\eps_1}}_\infty<1$  the Rayleigh quotient is larger than $-1$ on the space $\cL$ with
 $\codim \cL=n$.
 By the variational characterisation of the eigenvalues in \eqref{defkappan} we have $\kappa_{n+1}(\mu) > -1$ for all $\mu \in  ((\Lambda_1+1)^{-1}, (\Lambda_0+1)^{-1})$. Therefore,  using Lemma \ref{lemmavar}, we see that
 no more than $n$ eigenvalues of the operator $\lp $ are created in the gap.
\end{proof}

\section{Lower bound on the number of eigenvalues}\label{sec_LowerBound}

\begin{lemma} \label{lem8n} Let $\mu\in ((\Lambda_1+1)^{-1}, (\Lambda_0+1)^{-1})$. For all $u\in\ck$,
$$\llangle A_\mu u,u\rrangle_{\ck}\geq \dfrac{\sqrt{2\pi}}{1-\mu (\Lambda_1+1)} \norm{Ku}^2_{H^{-1}}.$$
holds.
\end{lemma}
\begin{proof}
As in the proof of Theorem \ref{upper}, we have that \eqref{rayleighlow} holds for $u\in\ck$. This leads to the estimate  
 \bea
\llangle A_\mu u,u\rrangle_{\ck}  
&\geq& \dfrac1{1-\mu (\Lambda_1+1)} \int_{-\pi}^\pi \sum_{s\geq s_0}   \dfrac1{\lambda_s(k)+1} \left|  \llangle\psi_s(k),\phi^{-1}_k (\hat V K u)_k\rrangle_{H^1(0,1)^2}\right|^2 dk\\
&\geq& \dfrac1{1-\mu (\Lambda_1+1)} \int_{-\pi}^\pi \sum_{s}   \dfrac1{\lambda_s(k)+1} \left|  \llangle\psi_s(k),\phi^{-1}_k (\hat V K u)_k\rrangle_{H^1(0,1)^2}\right|^2 dk\\
&\geq& \dfrac{\sqrt{2\pi}}{1-\mu (\Lambda_1+1)} \norm{Ku}_{H^{-1}}^2 
\eea
 where the last inequality follows from  Proposition \ref{prop:FB} \eqref{expansionlemma}.
\end{proof}

\begin{cor} \label{cor1} 
Let $\mu\in ((\Lambda_1+1)^{-1}, (\Lambda_0+1)^{-1})$ and suppose that $\norm{\frac{1}{\eps_0}-\frac{1}{\eps_1}}_\infty$  sufficiently small. Then 
 $$\inf_{u\in\ck\setminus\{0\}}\dfrac{  \llangle  A_\mu u,u  \rrangle}{\norm{u}^2_{\ck}}>-1.$$
\end{cor}
\begin{proof}
This follows from Lemma \ref{lem8n} together with Lemma \ref{lem12}.
\end{proof}

\begin{remark}
This shows that for a fixed $\mu$ in the spectral gap, the size of the perturbation has to reach a threshold before it is possible for $\mu$ to lie in the spectrum.
\end{remark}

\begin{theorem} \label{lower} Let Assumptions \ref{asseps} and \ref{asnc} hold. For any $\eps_1$ such that $\norm{\frac{1}{\eps_0}-\frac{1}{\eps_1}}_\infty$ is sufficiently small,   at least $n=|\Sigma|$ eigenvalues are created in the spectral gap.
\end{theorem}

\begin{proof}
By Corollary \ref{cor1},
 if   $\norm{\frac{1}{\eps_0}-\frac{1}{\eps_1}}_\infty$ is sufficiently small, we  can find $\mu'\in((\Lambda_1+1)^{-1},(\Lambda_0+1)^{-1})$ such that
\beq
\label{upperkappa} \kappa_1(\mu')=\inf_{u\neq0} \dfrac{  \llangle  A_{\mu'} u,u  \rrangle_{\ck}}{\norm{u}^2_{\ck}}>-1.
\eeq

We next give an upper bound on the Rayleigh quotient using equality \eqref{Rayleigh} and
 decomposing the sum over $s\in \N$ into three parts:
one over $s<s_0$, one over $s\geq s_0$ with $s\not\in S$, and one over $s\in S$. (Note that $s\geq s_0$ for all $s \in S$).
By Lemma \ref{lem:unif} the first sum is bounded from above by  $C\norm{u}^2_{\ck}$  as long as $\mu$ stays away from $(\Lambda_0+1)^{-1}$.  The second sum is bounded from above by $0$.
Therefore,
\bea
 \llangle A_\mu u, u \rrangle_{\ck}
   &\leq& C \norm{u}^2_{\ck} + \int_{-\pi}^\pi \sum_{s\in S} \dfrac1{1-\mu(\lambda_s(k)+1)}     \dfrac1{\lambda_s(k)+1} \left|  \llangle\psi_s(k),\phi^{-1}_k (\hat V K u)_k\rrangle_{H^1(0,1)^2}\right|^2 dk  
\eea
Now we split up the integration over $[-\pi, \pi]$ into a part over the intervals $B_\delta(k_j)$ and a remainder, as before
in the proof of Theorem
\ref{upper}.
We get 
\bea &&\llangle A_\mu u, u \rrangle_{\ck}\\ 
	&&\leq C \norm{u}^2_{\ck} +  \sum_{s\in S}  \left[\sum_{\stackrel{j=1}{s_j=s}}^n  \int_{B_\delta(k_j)} 
	\dfrac1{1-\mu(\lambda_s(k)+1)}     \dfrac1{\lambda_s(k)+1} \left|  \llangle\psi_s(k),\phi^{-1}_k (\hat V K u)_k\rrangle_{H^1(0,1)^2}\right|^2 dk 
	\right. \\
&&\left.  + \int_{R_s} 
	\dfrac1{1-\mu(\lambda_s(k)+1)}     \dfrac1{\lambda_s(k)+1} \left|  \llangle\psi_s(k),\phi^{-1}_k (\hat V K u)_k\rrangle_{H^1(0,1)^2}\right|^2 dk \right]
\eea
and using Lemma \ref{lem:unif} to estimate the integral over $R_s$, we  continue the estimate as follows:
\bea
	&&\leq C \norm{u}^2_{\ck} +  \sum_{s\in S}  \sum_{\stackrel{j=1}{s_j=s}}^n  \int_{B_\delta(k_j)} 
	\dfrac1{1-\mu(\lambda_s(k)+1)}     \dfrac1{\lambda_s(k)+1} \left|  \llangle\psi_s(k),\phi^{-1}_k (\hat V K u)_k\rrangle_{H^1}\right|^2 dk  
 \\
&&\leq C\norm{u}^2_{\ck} + \frac1n \sum_{j=1}^n\sum_{s\in S_{k_j}} \int_{B_\delta(0)} 
	\dfrac1{1-\mu(\lambda_s(k_j+\tilde{k})+1)}     \dfrac1{\lambda_s(k_j+\tilde{k})+1} \left|  \llangle\psi_s(k_j+\tilde{k}),\phi^{-1}_k (\hat V K u(k_j+\tilde{k}))\rrangle_{H^1}\right|^2 d\tilde{k} 
\\
&& \leq C\norm{u}^2_{\ck} +\frac1n \int_{B_\delta(0)} \frac{1}{1-\mu(\Lambda_1+\beta\tilde{k}^{2}+1)}f(\tilde{k},u)d\tilde{k}.
\eea
In the last but one inequality we use the fact that any $s\in S$ can be at most in $n$ sets $S_{{ k}_j}$; in the last line, due to analyticity, we have
for $|\tilde{k}|<\delta$ that $\lambda_s(k_j+\tilde{k})\leq \Lambda_1+\beta\tilde{k}^{2}$ for some $\beta>0$.
For any function
\begin{equation}
u=\sum_{i=1}^n \xi_i P\phi(\theta\Psi_i)\in\cL^\perp
\label{label}
\end{equation}
with coefficients $(\xi_i)_{i=1}^n\in\C^n$, we have from Lemma \ref{flowerbound} and continuity of $f$ that $f(\tilde k,u)$ is bounded below on $B_\delta(0)$. Thus the Rayleigh quotient satisfies the following estimate:
$$\frac{\llangle A_\mu u, u\rrangle_\ck}{\norm{u}_\ck^2}\leq C +\frac{C}{n} \int_{B_\delta(0)} \frac{1}{1-\mu(\Lambda_1+\beta\tilde{k}^{2}+1)} d\tilde{k} ~~~\hbox{ for some } C>0.$$
To show that the Rayleigh quotient tends to $-\infty$ as $\mu\to(\Lambda_1+1)^{-1}$, it is therefore sufficient for 
$$\int_0^\delta \frac{1}{\mu(\Lambda_1+\beta\tilde{k}^{2}+1)-1} d\tilde{k}$$ 
to diverge in the limit as $\mu\searrow(\Lambda_1+1)^{-1}$. We have
\bea
 \int_0^\delta \frac{1}{\mu(\Lambda_1+\beta\tilde{k}^{2}+1)-1} d\tilde{k}
&=& (\mu\beta(\mu(\Lambda_1+1)-1))^{-\frac12}\arctan\left(\delta\sqrt{\frac{\mu\beta}{\mu(\Lambda_1+1)-1}}\right)
\\
&\to & +\infty
 \hbox{ as } \mu\searrow(\Lambda_1+1)^{-1}.
\eea
Therefore, $$\max_{u\in\cL^\perp\setminus\{0\}} \dfrac{ \llangle A_\mu u, u \rrangle_{\ck}}{ \norm{u}^2_{\ck}} \to -\infty \hbox{ as } \mu\searrow(\Lambda_1+1)^{-1}.$$
As $\codim\cL=n$, the variational characterisation of the eigenvalues \eqref{defkappan} implies $\kappa_n(\mu) \to - \infty$ as $\mu\searrow(\Lambda_1+1)^{-1}$, and combined with Lemma \ref{lem:cont} and  \eqref{upperkappa} this means that at least $n$ eigenvalues are   created in the gap.
\end{proof}
Theorem \ref{upper}  and Theorem \ref{lower} together yield the following result.
\begin{theorem}  \label{theoremall} Let Assumptions \ref{asseps}, \ref{asnc} and \ref{as1} hold, i.e. 
\begin{enumerate}
\item[(i)]  $\eps_0,\eps_1\in L^\infty(\R^2)$.
\item[(ii)]
  $\eps_i\geq c_0>0$ for some constant $c_0$ and $i=0,1$.
\item[(iii)]The perturbation is nonnegative, i.e.~
\begin{align*}
\eps_1(\bx) - \eps_0(\bx) \geq 0.
\end{align*}
\item[(iv)] There exists a ball $D$ such that $\eps_1-\eps_0>0$ on $D$. 
\item[(v)] The band functions $\lambda_s$ are not constant as functions of  $k\in[-\pi,\pi]$.
\item[(vi)] There are $\alpha>0$   and $\delta>0$ such that   for all $(\hat s, \hat k) \in \Sigma$ and ${ k }\in [-\pi,\pi] $ satisfying  $|{  k} -{ \hat k}| \leq \delta$, 
$$\lambda_{\hat s} ({ k})\geq \Lambda_1+\alpha | {k}-{ \hat k}|^2.$$
\end{enumerate}
Moreover, let $\norm{\frac{1}{\eps_0}-\frac{1}{\eps_1}}_\infty>0$ be sufficiently small. Then the number of eigenvalues of the operator $\lp $ in the gap $(\Lambda_0,\Lambda_1)$ equals $n$, the finite number of solution pairs
 $(s,{k})$ of the equation $\Lambda_1=\lambda_s({k})$.
\end{theorem}

\section{Funding}
\thanks{
This work was supported by the British Engineering and Physical Sciences Research Council [EP/I038217/1 to B.M.B. and I.W.], the National Science Foundation [DMS 1412023, DMS-1614797, DMS-1810687 to V.H.] and the Deutsche Forschungsgemeinschaft [CRC 1173 to M.P.].

\appendix

\section{Proof of Proposition \ref{prop:FB}}\label{FB}
From \cite[Theorem 4.3 \& Theorem 4.7]{BHPW11}, we have $\sigma(\lo)=\overline{  \cup_k  \sigma(L_k)}$, as required for Proposition \ref{prop:FB} \eqref{spectra}.

For Proposition \ref{prop:FB} \eqref{l1}, let $v\in\ch$ be defined by $v(\cdot,k)=(L_k-\lambda)^{-1} (\hat{V}f)_k$. Then in the proof of  \cite[Theorem 4.3]{BHPW11} 
it is shown that $(\lo-\lambda)u=f$, where $u=V^{-1}v$.
Thus both sides of \eqref{*} equal $v(\cdot,k)$ and the statement is true.

To prove Proposition \ref{prop:FB} \eqref{resexpansionlemma} let
$f \in L^2(\Omega)$ and use the decomposition (see \cite{KuchmentBook, BHPW11}) 
$$f(:)=\frac1{\sqrt{2\pi}}\sum_{s=1}^\infty  \int_{-\pi}^\pi \llangle  Uf, \psi_s(\cdot, k)\rrangle_{L^2} \psi_s(:,k)dk,$$
where the series converges in $L^2(\Omega)$.
Thus  for $g\in H^{-1}_{qp}(\Omega)$, 
$$(\lo-\lambda)^{-1} g(:) = \frac1{\sqrt{2\pi}} \sum_{s=1}^\infty \int_{-\pi}^\pi  \llangle U(\lo-\lambda)^{-1} g, \psi_s(\cdot,k)\rrangle_{L^2} \psi_s(:,k) dk$$
and using Proposition \ref{prop:FB} \eqref{l1} and that $U\vert_{H^1_{qp}(\Omega)}=V$ we get
\begin{align}\label{FloqRes}
(\lo-\lambda)^{-1}g (:){}& =\frac1{\sqrt{2\pi}}\sum_{s=1}^\infty  \int_{-\pi}^\pi  \llangle (L_k-\lambda)^{-1} \hat V g ( k), \psi_s(\cdot,k)\rrangle_{L^2}\psi_s(:,k) dk.
\end{align}
Now, with $\phi_k=(L_k+1):\ch_k \to \ch'_k$, using \eqref{eq:phik} and that $\phi_k$ and $(L_k-\lambda)^{-1}$ commute we have
\begin{align*}
  \llangle (L_k-\lambda)^{-1} (\hat V g)_k, \psi_s(\cdot,k)\rrangle_{L^2} &  =  \llangle (L_k-\lambda)^{-1} (\hat V g)_k, \phi_k  \psi_s(\cdot,k)  \rrangle_{\ch_k'}  \\
 {}& =  \llangle  (\hat V g)_k, (L_k-\lambda)^{-1} \phi_k  \psi_s(\cdot,k)  \rrangle_{\ch_k'}  \\
  {}& = \dfrac 1 { \lambda_s(k)-\lambda}   \llangle  (\hat V g)_k,  \phi_k  \psi_s(\cdot,k)  \rrangle_{\ch_k'}  \\
 {}& = \dfrac 1 { \lambda_s(k)-\lambda}  (\hat V g)_k [  \psi_s (   \cdot,k)]. 
\end{align*}
Inserting this in \eqref{FloqRes} gives Proposition \ref{prop:FB} \eqref{resexpansionlemma}.
We next show Proposition \ref{prop:FB} \eqref{RQ}. 
 Noting that for $\lambda=\frac{1}{\mu}-1$ we have
 $$-\frac{1}{\mu}\frac{1}{\lambda_s(k)-\lambda}= \dfrac1{1-\mu (\lambda_s(k)+1)}$$ and using \eqref{notation},
 by Proposition \ref{prop:FB} \eqref{resexpansionlemma}, we have
\bea &&
\llangle -\dfrac1\mu (\lo+1-\dfrac1\mu)^{-1} (K u), Ku\rrangle_{H^{-1}} \\
  &&=\dfrac1{2\sqrt{\pi}}\llangle \sum_{s=1}^\infty\int_{-\pi}^\pi  \dfrac1{1-\mu (\lambda_s(k)+1)} ( \hat V Ku)_k  [\psi_s(\cdot,k)]\psi_s(:,k) dk , Ku(:)\rrangle_{H^{-1}} \\
  &&=\dfrac1{\sqrt{2\pi}}\lim_{l\to\infty}\llangle \sum_{s=1}^l\int_{-\pi}^\pi  \dfrac1{1-\mu (\lambda_s(k)+1)}\llangle (\phi^{-1}_\ch \hat V Ku)(\cdot,k) , \psi_s(\cdot,k)\rrangle_{H^1}\psi_s(:,k) dk , \phi^{-1} Ku(:)\rrangle_{L^2(\Omega)}.
\eea 
Next let 
\begin{align*}
 \chi_l(:,k)&= \sum_{s=1}^l \dfrac1{1-\mu (\lambda_s(k)+1)} \llangle (\phi^{-1}_\ch \hat V Ku)(\cdot,k) , \psi_s(\cdot,k)\rrangle_{H^1((0,1)^2)}\psi_s(:,k)
\end{align*}
Then using the formula \eqref{inverseFloquet} for the inverse Floquet transform and the isometry property of $U$ we get
\begin{align*} 
\llangle -\dfrac1\mu (\lo+1-\dfrac1\mu)^{-1} (K u), Ku\rrangle_{H^{-1}} 
&= \lim_{l\to\infty} \llangle \frac{1}{\sqrt{2\pi}}\int_{-\pi}^\pi E_k\chi_l(\cdot,k) \ dk, \phi^{-1} Ku\rrangle_{L^2(\Omega)}\\
&=\lim_{l\to\infty} \llangle U^{-1}\chi_l, \phi^{-1} Ku\rrangle_{L^2(\Omega)}\\
&=\lim_{l\to\infty}\llangle \chi_l, U\phi^{-1} K u\rrangle _{L^2( (0,1)^2\times [-\pi,\pi] ) }.
\end{align*}
Therefore, by Proposition \ref{prop:FB} \eqref{l1} using that $U\vert_{H^1_{qp}(\Omega)}=V$, and by \eqref{eq:phik} we get
\begin{align}\nonumber
 &\llangle -\dfrac1\mu (\lo+1-\dfrac1\mu)^{-1} (K u), Ku\rrangle_{H^{-1}} \\ \nonumber
 {}& = \lim_{l\to\infty}\int_{-\pi}^\pi \llangle \chi_l(k),\phi^{-1}_k (\hat V Ku)_k\rrangle_{L^2(0,1)^2} dk \\ \nonumber
 {}&= \lim_{l\to\infty}\int_{-\pi}^\pi    \llangle\phi^{-1}_k\chi_l(k),\phi^{-1}_k (\hat V Ku)_k\rrangle_{H^1(0,1)^2} dk\\ \nonumber
 {}&= \lim_{l\to\infty}\int_{-\pi}^\pi \overline{(\hat V Ku)_k[\phi^{-1}_k\chi_l(k)]}\\ \nonumber
  {}&= \lim_{l\to\infty} \int_{-\pi}^\pi\llangle\phi^{-1}_k\chi_l(k),  (\phi^{-1}_\ch \hat V K u)(k)\rrangle_{H^1(0,1)^2} \\ \nonumber
 {}& = \lim_{l\to\infty}\int_{-\pi}^\pi \sum_{s=1}^l \dfrac1{1-\mu(\lambda_s(k)+1)}\dfrac1{\lambda_s(k)+1}\llangle (\phi^{-1}_\ch \hat V K u)(k),\psi_s(k)\rrangle_{H^1(0,1)^2}\\ \nonumber
 {} & \hspace{175pt}
\llangle\psi_s(k),(\phi^{-1}_\ch \hat V K u)(k)\rrangle_{H^1(0,1)^2} dk.
\\ \label{eq:liml}
 {}&=\lim_{l\to\infty}\int_{-\pi}^\pi  \sum_{s=1}^l\dfrac1{1-\mu(\lambda_s(k)+1)}     \dfrac1{\lambda_s(k)+1} \left|  \llangle\psi_s(k),\phi^{-1}_k (\hat V K u)_k\rrangle_{H^1(0,1)^2}\right|^2 dk.
\end{align}
We now wish to interchange the order of taking the limit and integrating.
To do this note that
 \bea
 \chi_l(:,k)
 &=& \sum_{s=1}^l \dfrac{\lambda_s(k)+1}{1-\mu (\lambda_s(k)+1)} \llangle (\phi^{-1}_\ch \hat V Ku)(\cdot,k) , \frac{\psi_s(\cdot,k)}{\sqrt{\lambda_s(k)+1}}\rrangle_{H^1((0,1)^2)}\frac{\psi_s(:,k)}{\sqrt{\lambda_s(k)+1}} 
 \eea
and set $$\chi(:,k)= \sum_{s=1}^\infty \dfrac1{1-\mu (\lambda_s(k)+1)} \llangle (\phi^{-1}_\ch \hat V Ku)(\cdot,k) , \psi_s(\cdot,k)\rrangle_{H^1((0,1)^2)}\psi_s(:,k) .$$
Since the set $ \left\{  \dfrac{  \psi_s(k)} {\sqrt{ \lambda_s(k)+1}}\right\}$ is an orthonormal basis in $\ch_k$ and $(\phi^{-1}_\ch \hat V K u)(k)\in\ch_k$, the series converges in $\ch_k$. In particular, we have that for every $k\in[-\pi,\pi]$
 $$
\chi_l(\cdot,k)   \to \chi (\cdot,k)  \hbox{ in } H^1( (0,1)^2)    \; {\rm as \;} l\to \infty.
$$
Moreover, by Bessel's inequality
$$
\int_{(0,1)^2 }|  \chi (x,k)-\chi_l(x,k)|^2 dx \leq \int_{(0,1)^2 }| \chi (x,k)|^2 dx 
$$
and as a function of $k$ the right hand side lies in $L^1(-\pi,\pi).$
By Fubini's theorem, we have
\beq
\int_{-\pi}^\pi  \left (   \int _{(0,1)^2}  | \chi(x,k)-\chi_l(x,k)|^2 dx \right ) dk = \int_{(0,1)^2} \left ( \int_{-\pi}^\pi  | \chi (x,k)-\chi_l(x,k)|^2dk \right )  dx  \label{Fubini}
\enq
 and by dominated convergence the LHS of \eqref{Fubini} tends to $0$  and so the RHS  of \eqref{Fubini} also does.
 This implies that
\beq
\int_{-\pi}^\pi \chi_l(\cdot, k) dk \to\int _{-\pi}^\pi \chi(\cdot,k) dk \; \hbox{ in\;} L^2((0,1)^2).\label{**}
\eeq
Therefore, using the Cauchy-Schwarz inequality we have that
\begin{align*}
\left|\int_{-\pi}^\pi \llangle (\phi^{-1}_\ch \hat V K u)(k),\phi_k^{-1}(\chi-\chi_l)(k)\rrangle_{H^1(0,1)^2}dk\right| \\
\leq \int_{-\pi}^\pi \left|\llangle (\phi^{-1}_\ch \hat V K u)(k),(\chi-\chi_l)(k)\rrangle_{H^1(0,1)^2}\right| dk \to 0
\end{align*}
as $l\to \infty$ and so we can exchange the order of summation over $s$ and integration over $k$ in \eqref{eq:liml}.
  This gives
\begin{align*}
 &\llangle -\dfrac1\mu (\lo+1-\dfrac1\mu)^{-1} (K u), Ku\rrangle_{H^{-1}} \\
 {}&=\int_{-\pi}^\pi  \sum_{s=1}^\infty\dfrac1{1-\mu(\lambda_s(k)+1)}     \dfrac1{\lambda_s(k)+1} \left|  \llangle\psi_s(k),\phi^{-1}_k (\hat V K u)_k\rrangle_{H^1(0,1)^2}\right|^2 dk.
\end{align*}
  proving Proposition \ref{prop:FB} \eqref{RQ}.

Finally, we consider Proposition \ref{prop:FB} \eqref{expansionlemma}.
For $f\in H^{-1}_{qp}(\Omega)$, we have 
\bea
\norm{f}_{H^{-1}}^2 &=& \llangle \phi^{-1} f, \phi^{-1} f\rrangle_{H^1} \ =\ \llangle (\lo+1)^{-1} f, \phi^{-1} f\rrangle_{H^1} \\
&=&\frac1{\sqrt{2\pi}}\llangle \sum_{s=1}^\infty \int_{-\pi}^\pi \dfrac1{\lambda_s(k)+1} ( \hat V f)_k[\psi_s(\cdot, k)] \psi_s( :,k) dk, \phi^{-1} f(:)\rrangle_{H^1}
\eea
where we have used Proposition \ref{prop:FB} \eqref{resexpansionlemma} for $\lambda=-1$.

 Next, let $$\tilde\chi(:,k)=    \sum_{s=1}^\infty \dfrac1{\lambda_s(k)+1} \llangle (\phi^{-1}_\ch \hat V f)(\cdot,k) , \psi_s(\cdot,k)\rrangle_{H^1((0,1)^2)}\psi_s(:,k).$$
 
 By a similar argument to the proof of Proposition \ref{prop:FB} \eqref{RQ}, we can swap the order of summation and integration and 
 then using the formula \eqref{inverseFloquet} for the inverse Floquet transform and the isometry property of $V$ we get
\bea 
\norm{f}_{H^{-1}}^2  &=& \llangle \frac{1}{\sqrt{2\pi}}\int_{-\pi}^\pi E_k\tilde\chi(\cdot,k) \ dk, \phi^{-1} f\rrangle_{H^1} \ =\ \frac{1}{\sqrt{2\pi}}\llangle V^{-1}\tilde\chi,\phi^{-1}f\rrangle_{H^1} \ =\ \frac{1}{\sqrt{2\pi}} \llangle \tilde\chi, V\phi^{-1} f\rrangle _{\ch}.
\eea

 Therefore, using Proposition \ref{prop:FB} \eqref{l1}
\bea
 \norm{f}_{H^{-1}}^2 
 & = & \frac{1}{\sqrt{2\pi}}\int_{-\pi}^\pi \llangle \tilde\chi(k),\phi^{-1}_k (\hat V f)_k\rrangle_{\ch_k} dk \\
  & = & \frac{1}{\sqrt{2\pi}}\int_{-\pi}^\pi \llangle  \sum_{s=1}^\infty \dfrac1{\lambda_s(k)+1} \llangle (\phi^{-1}_\ch \hat V f)(\cdot,k) , \psi_s(\cdot,k)\rrangle_{H^1((0,1)^2)}\psi_s(:,k),\phi^{-1}_k (\hat V f)_k\rrangle_{\ch_k} dk \\
 &=&\frac{1}{\sqrt{2\pi}}\int_{-\pi}^\pi \sum_{s=1}^\infty   \dfrac1{\lambda_s(k)+1} \left|  \llangle\psi_s(k),\phi^{-1}_k (\hat V f)_k\rrangle_{H^1(0,1)^2}\right|^2 dk.
\eea
This completes the proof.

\end{document}